\documentclass[a4paper,times]{amsart}

\usepackage[pagebackref,colorlinks]{hyperref}
\usepackage[nameinlink]{cleveref}

\usepackage{amsthm,amsmath,amssymb}

\newtheorem{theorem}{\rm\bf Theorem}[section]
\newtheorem{proposition}[theorem]{\rm\bf Proposition}

\newtheorem{corollary}[theorem]{\rm\bf Corollary}

\newtheorem*{theorem*}{Theorem}
\newtheorem*{theorem 1}{\rm\bf Proposition 1}
\newtheorem*{theorem 2}{\rm\bf Proposition 2}

\theoremstyle{definition}
\newtheorem{definition}[theorem]{\rm\bf Definition}

\theoremstyle{remark}
\newtheorem{remark}[theorem]{\rm\bf Remark}

\def\half#1#2{\begin{matrix}\frac{#1}{#2}\end{matrix}}

\def\R#1{\mathbb{R}^{#1}}

\def\scal#1#2{\langle #1,\, #2 \rangle}

\def\field{k}

\DeclareMathOperator{\Idm}{Idm}

\DeclareMathOperator{\trace}{tr}

\begin{document}

\title{On an extremal property of  Jordan algebras of Clifford type}

\author{Vladimir G. Tkachev}
\address{Department of Mathematics, Link\"oping University\\ Link\"oping, 58183, Sweden, vladimir.tkatjev@liu.se}

\begin{abstract}
If $V$ is a finite-dimensional unital  commutative (maybe nonassociative) algebra carrying an associative positive definite bilinear form $\scal{}{}$ then there exist a nonzero idempotent $c\ne e$ ($e$ being the algebra unit) the shortest possible length $|c|^2:=\scal{c}{c}$. In particular,  $|c|^2\le \half12|e|^2$. We prove that the equality holds exactly when $V$ is  a Jordan algebra of Clifford type.
\end{abstract}

\keywords{
Commutative nonassociative algebras; Associative bilinear form; Idempotents; Jordan algebras; Jordan algebras of Clifford type
}
\subjclass[2000]{
Primary 17A99, 17C27}

\maketitle

\section{Introduction}
Throughout this paper, $V$ denotes a finite-dimensional commutative nonassociative algebra over $\R{}$ carrying an associative nonsingular bilinear form:
\begin{equation}\label{Qass}
\scal{xy}{z}=\scal{x}{yz}, \quad \forall x,y,z\in V,
\end{equation}
where $(x,y)\to xy=yx$ is the algebra multiplication.
Following M.~Bordemann \cite{Bordemann}, an algebra satisfying \eqref{Qass} is called \textit{metrised}. We shall always assume that $V$ is a \textit{Euclidean} metrised algebra, i.e. the associative bilinear form $\scal{\cdot}{\cdot}$ is positive definite. The classical example  is formal real (Euclidean) Jordan algebras with the invariant trace form \cite{FKbook}, \cite{Koecher}. Another important example is the Griess algebra $\mathfrak{G}$  appearing in connection with the Monster sporadic simple group \cite{Norton94} or, in general, any axial algebra \cite{HRS15}, \cite{Rehren17}, \cite{HRS15}, \cite{Ivanov15}. A more recent example which also appears in this context is the class of  nonassociative algebras decoding the geometric structure  of cubic minimal cones  and  cubic polynomial solutions to certain elliptic PDEs \cite[Chapter~6]{NTVbook}, \cite{NTV}, \cite{Tk14}, \cite{Tk15b}. Some related questions as well as geometry of idempotents are also discussed in \cite{RowenSegev}, \cite{HRS15b}, \cite{HSS17}, \cite{Castillo-Ramirez}.

It is known \cite{Tk15b}, \cite{KrTk18a} (see also \Cref{pro:main} below) that if $V$  is a Euclidean metrised algebra then the set of nonzero idempotents $\Idm(V)$ is nonempty and there exists an idempotent $c\ne0$ such that
\begin{equation}\label{cc1}
|c|^2\le |c'|^2,\qquad \forall c'\in \Idm(V),
\end{equation}
where $|x|^2=\scal{x}{x}$. Such an idempotent is called {\textit{extremal}}, denoted by $c\in \Idm_1(V)$. In other words, $\Idm_1(V)$ denotes the set of shortest nonzero idempotents in $V$.

If  additionally $V$ is a \textit{unital} algebra with the unit $e$ then the conjugation $c\to \bar c:=e-c$ is a natural involution on   $\Idm(V)$, indeed,
\begin{equation}\label{barc}
\bar c^2=e-2ec+c=e-c=\bar c,
\end{equation}
and similarly
$c\bar c=0.$
Therefore, using \eqref{Qass} one obtains
\begin{equation}\label{orthc}
\scal{c}{\bar c}=\scal{c^2}{\bar c}=\scal{c}{c\bar c}=0,
\end{equation}
i.e. $c$ and $\bar c$ are orthogonal. This together with \eqref{cc1} immediately yields
\begin{equation}\label{cce}
|e|^2=|c|^2+|\bar c|^2\ge 2|c|^2, \qquad \forall c\in \Idm_1(V).
\end{equation}

\begin{definition}\label{def:1}
A unital Euclidean metrised algebra such that the equality in \eqref{cce} attains for some $c\in \Idm_1(V)$ is said to be \textit{minimal}.
\end{definition}

It is the main purpose of the present paper to completely characterize the class of minimal algebras. Let us recall some standard definitions, see \cite[p.~13-14]{JacobsonBook}, \cite{FKbook}, \cite{McCrbook}. Let $U$ be a vector space over a field $\field$ of characteristic $\ne 2$, $f(u,v)$ be a nonsingular symmetric bilinear form from $U$ to $\field$. Then $\field \oplus U$ together with the multiplication law
\begin{equation}\label{jordanmul}
(a\oplus u)\bullet (b\oplus v)=(ab+f(u,v))\oplus (av+bu)
\end{equation}
is a Jordan algebra, called the Jordan algebra of bilinear form $f$ and denoted by $U_f$ (also known as Jordan algebra of Clifford type or a spin-factor \cite{McCrbook}). Then  $\epsilon=1\oplus 0$ is the unit in $U_f$ and any element  $z=a\oplus u \in U_f$ satisfies the quadratic relation
\begin{equation}\label{uquadrat}
z^2-z\trace z+d(z)\epsilon=0,
\end{equation}
where $\trace z=2a$ is the generic trace and $d(z)=a^2-f(u,u)$ is the generic norm of $z$. It is well-known that the trace form $t(z,w):=\trace (z\bullet w)$ is associative in the sense of \eqref{Qass}. In particular,  $U_f$ is a metrised algebra with respect to the trace form.

Our main result states that any minimal algebra is a  Jordan algebra of Clifford type. More precisely, we shall prove

\begin{theorem}\label{th:main}
If $V$ is a minimal algebra then
\begin{equation}\label{thspann}
V=\mathrm{span}(\Idm_1(V)).
\end{equation} Furthermore,  $V$ is isomorphic to the Jordan algebra $e^\bot_f$ of the symmetric bilinear form
$$
f(x,y)=\frac{1}{|e|^2}\scal{x}{y},
$$
where $e^\bot=\{x\in V:\scal{x}{e}=0\}$.
\end{theorem}

The paper is organized as follows. In \Cref{sec:extr} we give a short overview of metrised algebras and discuss variational properties of extremal idempotents in more details. The spectral inequality \eqref{extremal} in \Cref{pro:main} and \Cref{pro:norm2} are key ingredients in the classification of minimal algebras. The proof of Theorem~\ref{th:main} is given in \Cref{sec:last}.



\section{Extremal idempotents}\label{sec:extr}

\subsection{Euclidean metrised algebras}
Let $V$ be a finite dimensional algebra over $\R{}$ with multiplication denoted by juxtaposition $(x,y)\to xy\in V.$ A symmetric $\R{}$-bilinear form $\scal{x}{y}:V\times V\to \R{}$ is called nonsingular if $\scal{x}{y}=0$ for all $y\in V$ implies $x=0$. In what follows, we use the standard squared norm notation $$|x|^2=\scal{x}{x}.
$$
The bilinear form $\scal{\cdot}{\cdot}$ is called {associative} \cite{Schafer}, \cite{Okubo81a}  if \eqref{Qass} holds.
An algebra carrying  a non-singular symmetric bilinear form is called {metrised}, cf.  \cite{Bordemann}.

The operator of left ($=$right) multiplication is denoted by $L_x:y\to xy$. If algebra $V$ is metrised then $L_c$ is  self-adjoint:
\begin{equation}\label{selff}
\scal{L_xy}{z}=\scal{y}{L_xz}, \qquad \forall x,y,z\in V.
\end{equation}

We recall the following result, see also \cite{Tk15b}, \cite{Tk18a}.

\begin{proposition}\label{pro:main}
Let $(V,\scal{\cdot}{\cdot})$ be a nonzero Euclidean metrised algebra (i.e. $VV\ne 0$). Then the set $E$ of constrained stationary points of the variational problem
\begin{equation}\label{variational}
\scal{x}{x^2}\to \max\quad \text{ subject to a constraint}\quad \scal{x}{x}=1
\end{equation}
is nonempty and the maximum is attained. Denote by $E_0\subset E$ the (nonempty) set of local maxima in \eqref{variational}. Then for any $x\in E$, either $x^2=0$ or $c:=x/\scal{x^2}{x}$ is an idempotent in $V$. If additionally $x\in E_0$ then $\scal{x^2}{x}>0$ and the corresponding idempotent  $c=x/\scal{x^2}{x}$ satisfies the extremal property
\begin{equation}\label{extremal}
L_c\le \half12 \text{ on }c^\bot:=\{x\in V:\scal{x}{c}=0\}.\end{equation}
In particular, the eigenvalue $1$ of $L_c $ is simple.
\end{proposition}

\begin{proof}
By the positive definiteness of $\scal{\cdot}{\cdot}$, the unit sphere $S=\{x\in V:\scal{x}{x}=1\}$  is compact in the standard Euclidean norm topology on $V$ induced by $|x|^2$.
Since $VV\ne 0$, the cubic form $f(x)=\half16\scal{x^2}{x}\not\equiv0$ (this can be easily seen by polarization of $f$). Since $f$ is a continuous odd function on $S$, $f(x)$ attains its \textit{positive} maximum value there. In particular $E_0$ (and, thus, $E$) is nonempty.

First let $x\in E$ be a stationary point. Then
Lagrange’s undetermined multipliers method yields for any $y\in V$ and $\scal{x}{y}=0$ that the directional derivative $\partial_y f|_{x}$ vanishes:
$$
0=\partial_y f|_{x}=\half{1}{2}\scal{x^2}{y}.
$$
Since the inner product is nonsingular, this implies that $x$ and $x^2$ are parallel, therefore $x^2=\lambda x$ with $\lambda=\scal{x^2}{x}$. Note that $x^2=0$ if and only if $\scal{x^2}{x}=0$, otherwise scaling $x$ appropriately yields that $c:=x/\scal{x^2}{x}$ is  a nonzero idempotent.

Next, to prove \eqref{extremal}, let us consider an arbitrary $x\in E_0$, $|x|=1$, and define $x_t=x\cos t +y\sin t$, where $y\in V$ such that $\scal{x}{y}=0$, $|y|=1$ and $t\in \R{}$. Since $|x_t|=|x|=1$, we obtain by the extremal property of $x_0$ that
\begin{equation}\label{minimi}
\scal{x_t^2}{x_t}\le \scal{x_0^2}{x_0}= \scal{x^2}{x}.
\end{equation}
Using the associativity of the inner product we obtain
\begin{align*}
\scal{x_t^2}{x_t}&=\scal{x^2\cos^2 t +2xy\cos t\sin t+y^2\cos^2 t}{x\cos t +y\sin t}\\
&=\cos^3 t \scal{x^2}{x}+3\cos^2t\sin t\scal{x^2}{y}+3\cos t\sin^2 t\scal{x}{y^2}+\sin^3t\scal{y^2}{y}
\end{align*}
then using the fact that $\scal{x^2}{y}=\scal{\lambda x}{y}=0$, we find from \eqref{minimi}
\begin{align*}
3\cos t\sin^2 t\scal{x}{y^2}+\sin^3t\scal{y^2}{y}\le (1-\cos^3 t) \scal{x^2}{x},
\end{align*}
Dividing the obtained inequality by $\sin^2 t$ and passing to limit as $t\to 0$ yields
\begin{align*}
3\scal{x}{y^2}\le \half32\scal{x^2}{x}.
\end{align*}
Now, recall that $c=x/\scal{x}{x^2}$ is an extremal idempotent, hence
$$
\scal{y}{L_cy}=\scal{y}{cy}=\scal{c}{y^2}\le \half12=\half12|y|^2
$$
Thus, the latter inequality holds for all $y\bot c$ and $|y|=1$, implying by the self-adjointness of $L_c$ that $L_c\le \half12$ on $c^\bot$, as desired.
\end{proof}

We shall also need the following generalization.

\begin{corollary}
\label{cor:w}
Let $V$ be as in \Cref{pro:main} and let $V=U\oplus W$ be an orthogonal decomposition, where $U,W$ are nontrivial subspaces of $V$. If $UU\ne 0$ (i.e. $U$ is not zero subalgebra of $V$) then there exists $u\in U$, such that $u\ne 0$ and
$$
u^2=\lambda u+w, \quad \text{where }\,\,\lambda\in \R{}, \,\,\,w\in W.
$$
\end{corollary}

\begin{proof}
If $\dim U=1$ then the claim is trivial, therefore assume that $\dim U\ge 2$. Consider the variational problem of maximizing of the cubic form $f(x):=\scal{x^2}{x}$ under two conditions: $|x|^2=1$ and $x\in U$. If $U$ is a nonzero subalgebra of $V$ then $f(x)\not\equiv 0$ on $U$, hence arguing as in \Cref{pro:main} we conclude that the (positive) maximum attains at some $u\in U$,  $|u|=1$. Therefore for any $y\in U\cap u^\bot$ we have
$$
0=\partial_y f|_{u}=\half{1}{2}\scal{u^2}{y}.
$$
The latter implies that
$$
u^2\in (U\cap u^\bot)^\bot=\mathrm{span}(W \cup u),
$$
hence there exists $\lambda\in \R{}$ such that $u^2-\lambda u\in W$, the desired conclusion follows.
\end{proof}

\subsection{Variational properties}

It follows from \Cref{pro:main} that the set of nonzero idempotents is nonempty for any Euclidean metrised algebra. Moreover, any idempotent arising as in \Cref{pro:main} has distinguished spectral properties.

\label{sec:dim2}
\begin{definition}
An idempotent $c\in V$ is called \textit{extremal}, or $c\in \Idm_1(V)$, if the function $\scal{x^2}{x}$ attains its global maximum value at $c$ for all $x\in V$ satisfying $|x|=|c|$.

\end{definition}

It follows from the definition that if $c$ is an extremal idempotent then
\begin{equation}\label{ineq}
\scal{x}{x^2}\le \frac{1}{|c|}|x|^3, \qquad \forall x\in V,
\end{equation}
and the equality holds for $x=c$. Furthermore, all extremal idempotent have the same length. It also follows from \eqref{ineq} that
\begin{equation}\label{ineqcc}
|c|\le |c'|, \qquad \forall c\in \Idm_1(V),\,\, \forall c'\in \Idm(V),
\end{equation}
i.e. the extremal idempotents have the minimal possible squared length.

\begin{remark}\label{rem:jordan}
It is well-known, see for example \cite{FKbook}, that if $J$ is an Euclidean Jordan algebra equipped with the trace form $\scal{x}{y}=\trace(xy)$ then for any idempotent $c\in \Idm(J)$: $|c|^2=n\in \mathbb{Z}^+$, i.e. the squared length takes only positive integer values, and an idempotent $c\in J$ is primitive if and only it has minimal possible squared length. Furthermore, if $J$ is a spin factor (a Jordan algebra associated with symmetric bilinear form) then all (nonzero) idempotents have the same square length $|c|^2=1$.
\end{remark}

Taking into account the previous remark, it is convenient to scale  the inner product  such that all extremal idempotents has the unit length. To this end, note that if $\scal{\cdot}{\cdot}$ is an associative positive definite bilinear form on  $V$ then so also is $k\scal{\cdot}{\cdot}$ for any $k>0$. We have  the following definition.

\begin{definition}
Let $V$ be a Euclidean metrised algebra with an associative inner product $\scal{\cdot}{\cdot}$. The inner product is said to be \textit{normalized} if
\begin{equation}\label{ineq1}
\scal{x}{x^2}\le |x|^3
\end{equation}
holds for all $x\in V$ and the equality holds for some $x\ne 0$. By abuse of terminology, we call $V$ normalized if its inner product is so.
\end{definition}

In other words, the inner product is normalized if and only if any extremal idempotent has length $|c|=1$. Therefore, in an normalized algebra
$$
\Idm_1(V)=\{c\in V: c^2=c \quad \text{and}\quad |c|^2=1\}.
$$

As a corollary of \eqref{ineqcc},
\begin{equation}\label{ineq2}
|c|\ge 1, \quad \forall c\in \Idm(V).
\end{equation}
If the inner product normalized,  \Cref{pro:main} can be reformulated as follows.

\begin{proposition}
The equality in \eqref{ineq1} is obtained for $x\ne0$ if and only if $x/|x|\in \Idm_1(V)$. If $c\in \Idm_1(V)$ then $L_c\le \frac12$ on $c^\bot$.
\end{proposition}

\subsection{Unital algebras}
Recall that an algebra is called unital if there exists $e\in V$ such that $ex=xe=x$ for all $x$. If a unit exists then it is necessarily  unique and distinct from $0$.

\begin{proposition}\label{pro:1}
If $V$ is a  unital Euclidean metrised algebra and $\dim V\ge 2$ then there exist at least two different idempotents in $V$ distinct from the unit.
\end{proposition}

\begin{proof}
First note  that a unital algebra is obviously nonzero, because $ee=e\ne 0$, where $e\in V$ is the algebra unit. Therefore by \Cref{pro:main} there exists a (nonzero) extremal idempotent $c\in V$. By the extremal property, one also has $L_c\le \half12$ on $c^\bot$. Note that the subspace $c^\bot$ is nontrivial because by the assumption $\dim c^\bot=\dim V-1\ge 1$. The algebra unit  $e$ is an idempotent and $e\ne c$ because $L_e\equiv 1$ on $V$. By \eqref{barc}, $\bar c=e-c$ is also an idempotent and $c\ne e-c$ (because otherwise $e=2c$, hence $e=e^2=4c=2e$,  implying a contradiction). This proves the proposition.
\end{proof}

In the rest of this section we assume that $V$ is a unital algebra and $\dim V\ge 2$. Given an arbitrary idempotent $c\in \Idm(V)$, let $\bar c=e-c$ denote its conjugate. Then
$$
c\bar c=c(e-c)=c-c=0,
$$
and it follows from $c\bar c=0$ and $c$ and $\bar c$ are orthogonal, see \eqref{orthc}.

\smallskip
Recall that an idempotent distinct from $0$ and $e$ is called \textit{nontrivial}.
\smallskip

By the above, $\Idm_1(V)\ne \emptyset$. If $c\in \Idm_1(V)$ then $c^\bot$ is nontrivial, hence $c\ne e$ (since $L_e=1$ on the whole $V$). We have by the orthogonality and \eqref{ineq2}
\begin{equation}\label{unit2}
|e|^2=|c|^2+|\bar c|^2\ge 2.
\end{equation}
In the next section we classify all algebras where the equality in \eqref{unit2} is obtained.

\subsection{Minimal algebras}
For convenience reasons, we shall assume in the rest of the paper that $V$ is a normalized algebra. Then rephrasing the \Cref{def:1} yields that a unital normalized algebra  is \textit{minimal} if and only if
\begin{equation}\label{ee2}
|e|^2=2.
\end{equation}
Furthermore, specializing \eqref{ec} for   \eqref{ee2}  yields
\begin{equation}\label{ec1}
\scal{e}{c}=1, \quad \forall c\in \Idm_1(V).
\end{equation}
Since $|e|^2=2>1$, the unit $e$ is not an extremal idempotent. The next proposition shows that it is the only (distinct from zero)  non-extremal idempotent.

\begin{proposition}
\label{pro:norm2}
Let $V$ be a minimal algebra. Then all nonzero idempotents in $V$, except for the unit, are extremal:
$$
\Idm(V)=\Idm_1(V)\cup \{e\}.
$$
Furthermore, if $c\in \Idm_1(V)$ then
\begin{equation}\label{Lc12}
L_c=\half12 \quad\text{ on} \,\,c^\bot\cap \bar c^\bot.
\end{equation}
\end{proposition}

\begin{proof}
Let $c\in \Idm(V)$, $0\ne c\ne e$. Then
$$
2=|e|^2=|\bar c|^2+|c|^2\ge 2,
$$
therefore $|c|^2=1$, which implies that $\Idm(V)=\Idm_1(V)\cup \{e\}$. Next, if $c\in \Idm_1(V)$ then $\bar c\in \Idm_1(V)$ too. By the extremal property,
$$
L_c\le \half12 + c\otimes c, \quad
L_{\bar c}\le \half12 + \bar c\otimes \bar c.
$$
Since $L_{\bar c}=1-L_c$, this yields
$$
\half12 - \bar c\otimes \bar c\le L_c\le \half12 + c\otimes c
$$
which implies that $L_c=\half12$ on $c^\bot\cap \bar c^\bot$, as desired.
\end{proof}

The identity \eqref{Lc12} shows that the multiplication by an extremal idempotent is essentially $\half12$. This implies that the multiplication structure on a minimal algebra is quite special. More precisely we have

\begin{proposition}
If $V$ is a normalized  minimal algebra then for any $c_1,c_2\in \Idm_1(V)$ there holds
\begin{equation}\label{prod1}
2c_1c_2=c_1+c_2-\scal{\bar c_1}{c_2}e.
\end{equation}
\end{proposition}

\begin{proof}
First note that for any idempotent $c\in \Idm(V)$ there holds
\begin{equation}\label{ec}
\scal{e}{c}=\scal{e}{cc}=\scal{ec}{c}=|c|^2,
\end{equation}
hence if $c_1,c_2\in \Idm(V)$ then
\begin{equation}\label{conju}
\scal{c_1}{\bar c_2}=\scal{c_1}{e-c_2}=\scal{c_1}{e}-\scal{c_1}{c_2}=|c_1|^2-\scal{c_1}{c_2}.
\end{equation}

Next note that \eqref{prod1} trivially holds if $c_2=c_1$ or $c_2=\bar c_1$, therefore we assume that $c_2$ is  distinct from $c_1$ and $\bar c_1$. Let us decompose $c_2$ as
$$
c_2=\scal{c_1}{c_2}c_1+\scal{\bar c_1}{c_2}\bar c_1+z, \quad z\in c_1^\bot\cap {\bar c}_1^\bot.
$$
Since $L_{c_1}=\half12$ on $c_1^\bot\cap {\bar c}_1^\bot$ we have by virtue of \eqref{conju}
\begin{align*}
c_1c_2&=\scal{c_1}{c_2}c_1+\half12z=
\scal{c_1}{c_2}c_1+\half12(c_2-\scal{c_1}{c_2}c_1-\scal{\bar c_1}{c_2}\bar c_1)\\
&=\half12c_2+\half12\scal{c_1}{c_2}c_1-\half12\scal{\bar c_1}{c_2}(e-c_1)=\half12c_2+\half12c_1-\half12\scal{\bar c_1}{c_2}e,
\end{align*}
as desired.
\end{proof}

\begin{corollary}\label{cor:sub}
Let $S\subset \Idm_1(V)$ be such that $e\in \mathrm{span} (S)$, then $\mathrm{span}(S)$ is a subalgebra of $V$.
\end{corollary}

\begin{proof}
Follows readily from \eqref{prod1}.
\end{proof}

\section{The proof of the main results}\label{sec:last}
First we establishes \eqref{thspann} in \Cref{th:main}.

\begin{proposition}
\label{th:span}
If $V$ is a minimal algebra then $
V=\mathrm{span}(\Idm_1(V)).
$
\end{proposition}

\begin{proof}
Without loss of generality we may assume that $V$ is a normalized minimal algebra. Recall that $\Idm_1(V)$ is nonempty by \Cref{pro:main}. Define
$$
W:=\mathrm{span}(\Idm_1(V)).
$$
Then $W$ is a subalgebra by Corollary~\ref{cor:sub}. Since $c+\bar c=e$ for any $c\in \Idm_1(V)$, we have $e\in W$.
Assume by contradiction that  $W\ne V$, hence $V=W\oplus W^\bot$ with $W^\bot\ne 0$. Since $WW\subset W$, the associativity of the inner product implies that $WW^\bot\subset W^\bot$. Given an arbitrary $z\in W^\bot$, let us consider the orthogonal decomposition
\begin{equation}\label{xyz}
z^2=x+y, \quad x\in W, y\in W^\bot.
\end{equation}
Let $c\in \Idm_1(V)$ be chosen arbitrarily. By Proposition~\ref{pro:norm2},
\begin{equation}\label{halff}
L_c=\half12\quad  \text{ on }\quad W^\bot,
\end{equation}
hence $$
\scal{z^2}{c}=\scal{z}{zc}=\half12|z|^2,
$$
therefore it follows  from \eqref{xyz} that $\scal{x}{c}=\half12|z|^2$. 
Let $x_0=x-\half12|z|^2e$. Then $x_0\in W$ and using \eqref{ec1} we obtain
$$
\scal{x_0}{c}=\scal{x}{c}-\half12|z|^2\scal{e}{c} =\half12|z|^2-\half12|z|^2=0,
$$
which by virtue of the arbitrariness of $c\in \Idm(V)$ and the definition of $W$ implies that $x_0\in W^\bot$. Therefore $x_0\in W^\bot\cap W$, i.e. $x_0=0$. This proves that for any $z\in W^\bot$
\begin{equation}\label{xyz1}
z^2=\half12|z|^2e+y, \quad y\in W^\bot.
\end{equation}
In particular, $z^2\ne0$ if $z\ne 0$, hence $W^\bot W^\bot\ne 0$. This implies by Corollary~\ref{cor:w} that there exists a nonzero vector $z\in W^\bot$ such that
\begin{equation}\label{xyz2}
z^2=\half12|z|^2e+\lambda z,\quad \lambda\in \R{}.
\end{equation}
With this $z$ in hand, we claim that the idempotent equation $p^2=p$ for $p=ae+bz$ with $a,b$ being some real numbers has a  solution distinct from $e$. Note that we may assume that $b\ne 0$, because otherwise $p=e$. Then expanding $p^2=p$ by virtue of  \eqref{xyz2} yields the system
$$
\half12b^2|z|^2+a^2=a \quad \text{and}\quad 2ab+\lambda b^2=b.
$$
From the second equation we have $\lambda b=1-2a$, therefore
$$
\half12b^2|z|^2=a-a^2=\half14-(a-\half12)^2=\half14-\half14\lambda^2b^2.
$$
This yields $b^2=1/(2|z|^2+\lambda^2)$. Since the latter equation is solvable, this proves that there exists an idempotent $p=ae+bz\in \Idm(V)$ with $b\ne 0$, i.e. $p\ne e$. Therefore by \Cref{pro:norm2}, $p\in \Idm_1(V)\subset W$, which  obviously contradicts to the assumption that $0\ne z\in W^\bot$. This proves that assumption $W^\bot\ne0$ is wrong, thus $W=V$. The proposition follows.
\end{proof}

Now we are ready to finish the proof of \Cref{th:main}.
\begin{proposition}
\label{cor:idem}
If $V$ is a minimal algebra then $V$ is a Jordan algebra of the symmetric bilinear form
$$
f(x,y)=\frac{1}{|e|^2}\scal{x}{y}.
$$
\end{proposition}

\begin{proof}
First let $V$ be a normalized minimal algebra.
Note that  \eqref{thspann} readily implies  that there exists a basis in $V$ consisting of $e$ and idempotents with the unit norm. Let $\{e,c_1,\ldots,c_n\}$ be such a basis, where $n+1=\dim V\ge2$. It is easy to see that if $c_i$ is in the basis then $\bar c_i=e-c_i$ is not. This implies that the new set $\{e,e_1,\ldots,e_n\}$ with $e_i=c_i-\bar c_i=2c_i-e$, $1\le i\le n$, is also an basis of $V$.

Let $x\in V$ and let
\begin{equation}\label{xdecom}
x=a_0e+\sum_{i=1}^na_ie_i
\end{equation}
be the corresponding basis decomposition. By \eqref{ec1}
$$
\scal{e}{e_i}=\scal{e}{c-\bar c}=1-1=0
$$
for all $i\ge 1$, hence
$$
a_0=\frac{\scal{x}{e}}{\scal{e}{e}}=\frac12 \scal{x}{e}.
$$
Note also that for all $i\ge 1$
\begin{equation}\label{square}
e_i^2=4c-4c+e=e,
\end{equation}
and also by virtue of \eqref{prod1}
\begin{equation}\label{square1}
e_ie_j=4c_ic_j-2c_i-2c_j+e=(1-2\scal{\bar c_i}{c_j})e, \quad 1\le i,j\le n.
\end{equation}
Rewrite  \eqref{xdecom} as $x=\frac12 \scal{x}{e}e+y$, where $y=\sum_{i=1}^na_ie_i$. Since $y$ is orthogonal to $e$ we find that $|x|^2=\frac12 \scal{x}{e}^2+|y|^2$. Also from \eqref{square} and \eqref{square1} follows that $y^2=\lambda e$ for some $\lambda$, hence
$$
y^2=\frac{\scal{y^2}{e}}{\scal{e}{e}}e=\frac{|y|^2}{2}=\frac12(|x|^2-\half12 \scal{x}{e}^2)e.
$$
Since $(x-a_0e)^2=y^2,$ we obtain
$$
x^2-2a_0x+a_0^2e=\frac12(|x|^2-\half12 \scal{x}{e}^2)e.
$$
implying
\begin{equation}\label{quadr0}
x^2-\scal{x}{e}x+d(x)e=0,
\end{equation}
where
$$
d(x):=\half12(\scal{x}{e}^2-|x|^2).
$$
In particular, \eqref{quadr0} implies for the commutator
$$
[L_{x^2},L_x]=[L_{\scal{x}{e}x-b(x)e},L_x] =\scal{x}{e}[L_{x},L_x]-b(x)[1,L_x]=0
$$
implying that $V$ is a Jordan algebra. Polarizing \eqref{quadr0} yields the following explicit expression for the algebra multiplication:
\begin{equation}\label{explicitm}
xy=\half12\bigl(\scal{x}{e}y+\scal{y}{e}x -(\scal{x}{e}\scal{y}{e}-\scal{x}{y})e\bigr)
\end{equation}

It remains to establish an isomorphism between $V$ and the Jordan algebra $e^\bot_f$, where $f(x,y)=\half12\scal{x}{y}$ (recall that in a normalized minimal algebra $|c|^2=2$). To this end, note that $e^\bot_f=\R{}\oplus e^\bot$ with the product given by \eqref{jordanmul}, i.e. the present notation
$$
(a\oplus u)\bullet (b\oplus v)=(ab+\half12\scal{u}{v})\oplus (av+bu), \qquad u,v\in e^\bot,\,\, a,b\in \R{}.
$$
Let us define a homomorphism of vector spaces
$$
\phi(z)=a e+u:e^\bot_f\to V, \quad \text{where }\,\,z=a\oplus u.
$$
Then for $z=a\oplus u$ and $w=b\oplus v$ we have
$$
\phi(z\bullet w)=
(ab+\half12\scal{u}{v})+ (av+bu).
$$
On the other hand, using
$$
\scal{\phi(z)}{e}=2a, \quad
\scal{\phi(w)}{e}=2b, \quad
\scal{\phi(z)}{\phi(w)}=2ab+\scal{u}{v}, \quad
$$
and applying \eqref{explicitm} we find
\begin{align*}
\phi(z)\phi(w)&=\half12\bigl(2a\phi(w)+2b\phi(z)-(2ab-\scal{u}{v})e\bigr)\\
&=av+bu+(ab+\half12\scal{u}{v})e\\
&=\phi(z\bullet w),
\end{align*}
hence $\phi$ is an algebra isomorphism.

Finally, let $V$ be an arbitrary minimal algebra with inner product $\scal{\cdot}{\cdot}$. Define the new inner product by
$$
\scal{x}{y}_1=\frac{2\scal{x}{y}}{\scal{e}{e}},
$$
so that $\scal{e}{e}_1=2$. Then $(V,\scal{\cdot}{\cdot}_1)$ is obviously a \textit{normalized} minimal algebra. Thus
$$
V\cong e^{\bot}_{f}, \quad \text{where} \quad f=\half12 \scal{\cdot}{\cdot}_1=\frac{\scal{\cdot}{\cdot}}{\scal{e}{e}},
$$
as desired. The proposition is proved completely.
\end{proof}


\bibliographystyle{plain}

\begin{thebibliography}{10}

\bibitem{Bordemann}
M.~Bordemann.
\newblock Nondegenerate invariant bilinear forms on nonassociative algebras.
\newblock {\em Acta Math. Univ. Comenian.}, 66(2):151--201, 1997.

\bibitem{Castillo-Ramirez}
A.~Castillo-Ramirez.
\newblock Associative subalgebras of low-dimensional {M}ajorana algebras.
\newblock {\em J. Algebra}, 421:119--135, 2015.

\bibitem{FKbook}
J.~Faraut and A.~Kor{\'a}nyi.
\newblock {\em Analysis on symmetric cones}.
\newblock Oxford Mathematical Monographs. The Clarendon Press Oxford University
  Press, New York, 1994.
\newblock Oxford Science Publications.

\bibitem{HRS15}
J.~I. Hall, F.~Rehren, and S.~Shpectorov.
\newblock Primitive axial algebras of {J}ordan type.
\newblock {\em J. Algebra}, 437:79--115, 2015.

\bibitem{HRS15b}
J.~I. Hall, F.~Rehren, and S.~Shpectorov.
\newblock Universal axial algebras and a theorem of {S}akuma.
\newblock {\em J. Algebra}, 421:394--424, 2015.

\bibitem{HSS17}
J.~I. Hall, Y.~Segev, and S.~Shpectorov.
\newblock Miyamoto involutions in axial algebras of jordan type half.
\newblock {\em Israel Journal of Mathematics}, Nov 2017.

\bibitem{Ivanov15}
A.~A. Ivanov.
\newblock Majorana representation of the {M}onster group.
\newblock In {\em Finite simple groups: thirty years of the atlas and beyond},
  volume 694 of {\em Contemp. Math.}, pages 11--17. Amer. Math. Soc.,
  Providence, RI, 2017.

\bibitem{JacobsonBook}
N.~Jacobson.
\newblock {\em Structure and representations of {J}ordan algebras}.
\newblock American Mathematical Society Colloquium Publications, Vol. XXXIX.
  American Mathematical Society, Providence, R.I., 1968.

\bibitem{Koecher}
M.~Koecher.
\newblock {\em The {M}innesota notes on {J}ordan algebras and their
  applications}, volume 1710 of {\em Lecture Notes in Mathematics}.
\newblock Springer-Verlag, Berlin, 1999.
\newblock Edited, annotated and with a preface by Aloys Krieg and Sebastian
  Walcher.

\bibitem{KrTk18a}
Ya. Krasnov and V.G. Tkachev.
\newblock Variety of idempotents in nonassociative algebras.
\newblock In {\em Modern Trends in Hypercomplex Analysis}, Trends Math.
  Birkh\"auser/Springer, Cham, 2018, \textit{to appear}

\bibitem{McCrbook}
K.~McCrimmon.
\newblock {\em A taste of {J}ordan algebras}.
\newblock Universitext. Springer-Verlag, New York, 2004.

\bibitem{NTV}
N.~Nadirashvili, V.G. Tkachev, and S.~Vl{\u{a}}du{\c{t}}.
\newblock A non-classical solution to a {H}essian equation from {C}artan
  isoparametric cubic.
\newblock {\em Adv. Math.}, 231(3-4):1589--1597, 2012.

\bibitem{NTVbook}
N.~Nadirashvili, V.G. Tkachev, and S.~Vl{\u{a}}du{\c{t}}.
\newblock {\em Nonlinear elliptic equations and nonassociative algebras},
  volume 200 of {\em Mathematical Surveys and Monographs}.
\newblock American Mathematical Society, Providence, RI, 2014.

\bibitem{Norton94}
S.~Norton.
\newblock The {M}onster algebra: some new formulae.
\newblock In {\em Moonshine, the {M}onster, and related topics ({S}outh
  {H}adley, {MA}, 1994)}, volume 193 of {\em Contemp. Math.}, pages 297--306.
  Amer. Math. Soc., Providence, RI, 1996.

\bibitem{Okubo81a}
S.~Okubo and J.M. Osborn.
\newblock Algebras with nondegenerate associative symmetric bilinear forms
  permitting composition.
\newblock {\em Comm. Algebra}, 9(12):1233--1261, 1981.

\bibitem{Rehren17}
F.~Rehren.
\newblock Generalised dihedral subalgebras from the {M}onster.
\newblock {\em Trans. Amer. Math. Soc.}, 369(10):6953--6986, 2017.

\bibitem{RowenSegev}
L.~Rowen and Y.~Segev.
\newblock Associative and {J}ordan {A}lgebras {G}enerated by {T}wo
  {I}dempotents.
\newblock {\em Algebr. Represent. Theory}, 20(6):1495--1504, 2017.

\bibitem{Schafer}
R.D. Schafer.
\newblock {\em An introduction to nonassociative algebras}.
\newblock Pure and Applied Mathematics, Vol. 22. Academic Press, New York,
  1966.

\bibitem{Tk14}
V.G. Tkachev.
\newblock A {J}ordan algebra approach to the cubic eiconal equation.
\newblock {\em J. of Algebra}, 419:34--51, 2014.

\bibitem{Tk15b}
V.G. Tkachev.
\newblock On the non-vanishing property for real analytic solutions of the
  {$p$}-{L}aplace equation.
\newblock {\em Proc. Amer. Math. Soc.}, 144(6):2375--2382, 2016.

\bibitem{Tk18a}
V.G. Tkachev.
\newblock A correction of the decomposability result in a paper by
  {M}eyer-{N}eutsch.
\newblock {\em arXiv:1801.02339}, 2018.
\newblock submitted.

\end{thebibliography}

\def\cprime{$'$} \def\cprime{$'$}

\end{document}